\documentclass[a4paper,10pt]{article}
\usepackage{mathrsfs}
\usepackage{latexsym}
\usepackage{amsmath,amssymb}
\usepackage[pdftex,colorlinks]{hyperref}
\usepackage{amsthm}
\usepackage{amsfonts}
\usepackage[usenames]{color}
\usepackage{amssymb}
\usepackage{graphicx}
\usepackage{amsmath}
\usepackage{amsfonts}
\usepackage{amsthm}
\usepackage{mathrsfs}
\usepackage{dsfont}
\usepackage{indentfirst}
\usepackage[numbers,sort&compress]{natbib}
\usepackage{esint}

\newcommand\e{\varepsilon}

\newtheoremstyle{mythm}{1.5ex plus 1ex minus .2ex}{1.5ex plus 1ex
minus .2ex}{\kai}{\parindent}{\song\bfseries}{}{1em}{}
\numberwithin{equation}{section}

\newtheorem{theorem}{Theorem}[section]
\newtheorem{lemma}{Lemma}[section]

\newtheorem{corollary}{Corollary}[section]
\allowdisplaybreaks[4]
\begin{document}
\title{{\textbf{Uniqueness of topological solutions of  self-dual Chern-Simons equation with collapsing vortices}}}

\author{Genggeng Huang\footnote{genggenghuang@sjtu.edu.cn}\quad and Chang-Shou Lin\footnote{cslin@math.ntu.edu.tw}}

\date{}
\maketitle
\begin{center}{\footnotesize{$^*$ Department of Mathematics, INS and MOE-LSC, Shanghai Jiao Tong University, Shanghai, China}}\\
{\footnotesize{$^\dag$  Taida Institute for Mathematical Sciences, Center for Advanced Study in Theoretical Sciences,\\
National Taiwan University, Taipei, 10617, Taiwan}}\end{center}
\begin{abstract}
We consider the following Chern-Simons equation,
\begin{equation}
\label{0.1}
\Delta u+\frac 1{\varepsilon^2} e^u(1-e^u)=4\pi\sum_{i=1}^N \delta_{p_i^\varepsilon},\quad \text{in}\quad \Omega,
\end{equation}
where $\Omega$ is a 2-dimensional flat torus, $\varepsilon>0$ is a coupling parameter and $\delta_p$ stands for the Dirac measure concentrated at $p$. In this paper, we proved that  the topological solutions of \eqref{0.1} are uniquely determined by the location of their vortices provided the coupling parameter $\varepsilon$ is small and the collapsing velocity of vortices $p_i^\varepsilon$ is slow enough or fast enough comparing with $\varepsilon$. This extends the uniqueness results of Choe \cite{Choe2005} and Tarantello \cite{Tarantello2007}. Meanwhile, for any topological solution $\psi$ defined in $\mathbb R^2$ whose linearized operator is non-degenerate, we construct a sequence topological solutions $u_\varepsilon$ of \eqref{0.1} whose asymptotic limit is exactly $\psi$ after rescaling around $0$. A consequence  is  that non-uniqueness of topological solutions in $\mathbb R^2$ implies non-uniqueness of topological solutions on torus with collapsing vortices.
%\par {\bf 2010 Mathematics Subject Classification. 35J60,35J57
%\par Keywords: topological solutions; uniqueness;  Chern-Simons equation}
\end{abstract}

\section{Introduction}
This paper is devoted to study the following semi-linear elliptic equation with exponential nonlinearity,
\begin{equation}
\label{1.1}
\Delta u+\frac 1{\varepsilon^2} e^u(1-e^u)=4\pi\sum_{i=1}^N \delta_{p_i^\varepsilon},\quad \text{in}\quad \Omega,
\end{equation}
where $\Omega$ is a 2-dimensional flat torus, $\varepsilon>0$ is a coupling parameter and $\delta_p$ stands for the Dirac measure concentrated at $p$. Through out the paper, we always normalize the volume of $\Omega$ as $|\Omega|=1$.
\par  \eqref{1.1} arises in the Abelian Chern-Simons model introduced by Jackiw-Weinberg \cite{JackiwWeinberg1990} and Hong-Kim-Pac \cite{HongKimPac1990}. This model is given in the $(2+1)-$dimensional Minkowski space with metric $g_{\mu\nu}=diag(1,-1,-1)$. The Chern-Simons model in \cite{HongKimPac1990}  and \cite{JackiwWeinberg1990} can be formulated according to the Lagrangian density,
\begin{equation}
\label{1.2}
\mathcal L=\frac \kappa4\varepsilon^{\mu\nu\rho}F_{\mu\nu}A_\rho+D_\mu\phi\overline{D^\mu \phi}-\frac 1{\kappa^2}|\phi|^2(1-|\phi|^2)^2
\end{equation}
where $A_\mu$ $(\mu=0,1,2)$ is a real gauge field on $\mathbb R^{3}$, $\phi$ is the complex-valued Higgs field, $F_{\mu\nu}=\partial_\mu A_\nu-\partial_\nu A_\mu$ is the curvature tensor, $D_\mu=\partial_\mu-\sqrt{-1}A_\mu$ is the gauge covariant derivative, $\varepsilon^{\mu\nu\rho}$ is totally skew symmetric tensor with $\varepsilon^{012}=1$, and $\kappa>0$ is the Chern-Simons coupling constant.  When the energy for a pair $(\phi,A)$ is saturated, as in \cite{JackiwWeinberg1990} and \cite{HongKimPac1990}, one can get the following Bogomol'nyi type equation.
\begin{equation}
\label{1.3}
\begin{cases}
(D_1+iD_2)\phi=0,\\
F_{12}+\frac 2\kappa |\phi|^2(|\phi|^2-1)=0.
\end{cases}
\end{equation}
As in Jaffe-Taubes \cite{JaffeTaubes1980}, we let $u=\ln |\phi|^2$, and denote the zeros of $\phi$ by $\{p_1^\varepsilon,\cdots,p_N^\varepsilon\}$,   \eqref{1.3} can be transformed to \eqref{1.1} with $\varepsilon=\frac \kappa2$, if we impose the periodic boundary condition(introduced by 't Hooft \cite{Hooft1979}). For the details of derivation of \eqref{1.1} and related models, we refer the readers to Hong-Kim-Pac \cite{HongKimPac1990}, Jackiw-Weinberg \cite{JackiwWeinberg1990}, Dunne \cite{Dunne1995}, Tarantello \cite{Tarantello2008} and  Yang \cite{Yang2001}.
 \par A sequence of  solutions  $u_\varepsilon$ of \eqref{1.1} are called  topological type if
 \begin{equation*}
 u_\varepsilon(x)\rightarrow 0, \quad \text{a.e. in }\Omega, \quad \text{ as } \varepsilon\rightarrow 0,
 \end{equation*}
 and called non-topological type if
 \begin{equation*}
 u_\varepsilon(x)\rightarrow -\infty, \quad \text{a.e. in }\Omega, \quad \text{ as }\varepsilon\rightarrow 0.
 \end{equation*}
 When the vortices don't change with $\varepsilon$, we rewrite \eqref{1.1} as
 \begin{equation}
 \label{1.4}
 \Delta u+\frac 1{\varepsilon^2} e^u(1-e^u)=4\pi\sum_{i=1}^N \delta_{p_i},\quad \text{in}\quad \Omega.
 \end{equation}
\par The construction of topological vortex condensate $u(x)$ of \eqref{1.4} was first done by Caffarelli-Yang \cite{CaffarelliYang1995} via both monotone scheme and variational method. Then Tarantello \cite{Tarantello1996} further exploited the variational structure and got both topological (for general $N$) and non-topological vortex condensates(for $N=1$). After that, many papers were devoted to find the non-topological vortex condensate for $N\geq 2$. For these developments, we refer readers to \cite{NolascoTarantello1998,NolascoTarantello1999,DingJostLiWang1998,DingJostLiWang1999,linYan2010,linYan2013,Choe2009,ChoeKim2008} and references therein. All these works reveal that the non-topological solution isn't unique. As for the topological solutions, Choe \cite{Choe2005} and Tarantello \cite{Tarantello2007} independently proved that the topological solution of \eqref{1.4}  is unique when the coupling constant $\varepsilon>0$ is small enough. Their results can be summarized as follows.
 \\
 {\bf Theorem A.}  {\it There is a critical value of $\varepsilon$, say, $\varepsilon(p_1,\cdots,p_N)>0$ such that, for $0<\varepsilon<\varepsilon(p_1,\cdots,p_N)$,  \eqref{1.4} admits a unique topological solution.
}
\par By Theorem A, we shall see that the critical value $\varepsilon(p_1,\cdots,p_N)$ doesn't only depend on the vortex number $N$ but also on the location of the vortices on the flat torus $\Omega$. For the physical applications, it is relevant to know to what extend the uniqueness property stated above depends on the smallness of the parameter $\varepsilon$, and hence on the location of the vortex points. In the end of the paper of Tarantello \cite{Tarantello2007}, the author considered a generalization of Theorem A, where the vortex points are allowed to vary with $\varepsilon$ but no collapsing of vortices happens.  So a natural question is whether the critical value $\varepsilon(p_1,\cdots,p_N)$ in Theorem A depends only on the vortex number $N$. We give a partial answer to this problem in this paper. First, we classify the points $\{p_1^\varepsilon,\cdots,p_N^\varepsilon\}$ according to their asymptotic behavior. Along a subsequence, we define
\begin{equation}
\label{class1}
A_{k,\varepsilon}=\{p_{i}^{\varepsilon}|\lim_{\varepsilon\rightarrow 0}\frac{|p_{i}^{\varepsilon}-p_{k}^{\varepsilon}|}{\varepsilon}<+\infty\}.
\end{equation}
It follows directly from the definition of $A_{k,\varepsilon}$ that for any $k,m$, either $A_{k,\varepsilon}=A_{m,\varepsilon}$ or $A_{k,\varepsilon}\cap A_{m,\varepsilon}=\emptyset$. Without loss of generality, we can take $A_{i,\varepsilon},$ $i=1,\cdots,l$ such that
\begin{equation}
\label{class2}
A_{i,\varepsilon}\cap A_{j,\varepsilon}=\emptyset,\quad i\neq j,\quad \cup_{i=1}^lA_{i,\varepsilon}=\{p_{1}^{\varepsilon},\cdots,p_{N}^{\varepsilon}\}.
\end{equation}
%\begin{theorem}
%\label{thm1.1}
%$u^{\max}_{\bf p^\varepsilon}\rightarrow 0$ a.e. in $\Omega$.
%\end{theorem}

\begin{theorem}
\label{thm1.2}
There exists $C(N)>0$ such that if for any set $A_{i,\varepsilon}$, $i=1,\cdots,l$ defined in \eqref{class1}, \eqref{class2}, either
\begin{equation}
 \label{1.5}
 \begin{split}
 &(a.)\quad \lim_{\varepsilon\rightarrow 0}\frac{|p^\varepsilon_k-p^\varepsilon_m|}{\varepsilon}\geq C(N),\quad \text{for all different points}\quad p_{k}^{\varepsilon},p_{m}^{\varepsilon}\in A_{i,\varepsilon},\quad \text{or}\\
 &(b.)\quad \lim_{\varepsilon\rightarrow 0}\frac{|p^\varepsilon_k-p^\varepsilon_m|}{\varepsilon}\leq \frac1{C(N)},\quad \text{for all different points}\quad p_{k}^{\varepsilon},p_{m}^{\varepsilon}\in A_{i,\varepsilon}
 \end{split}
 \end{equation}
 hold. Then equation \eqref{1.1} has a unique topological solution for all small $\varepsilon>0$.
\end{theorem}
The constant $C(N)$ in Theorem \ref{thm1.2} is determined by the following theorem which is due to Choe \cite{Choe2005}.
\\
 {\bf Theorem B.} \cite{Choe2005} {\it For any $N>0$, there exists $C(N)>0$, such that $\sum_{j=1}^l\alpha_j\leq N$ and given $\{p_1\cdots,p_l\}$, either
 \begin{equation}
 \label{1.7}
 \begin{split}
 &(a.)\quad |p_i-p_j|\geq C(N),\quad\forall i\neq j,\quad \text{or}\\
 &(b.)\quad |p_i-p_j|\leq \frac 1{C(N)},\quad\forall i\neq j.
 \end{split}
 \end{equation}
 Then the equation
 \begin{equation}
 \label{1.6}
 \Delta u+e^u(1-e^u)=4\pi\sum_{i=1}^l\alpha_i \delta_{p_i},\quad \text{in}\quad \mathbb R^2,
 \end{equation}
 has a unique topological solution. Moreover, the linearized operator
 \begin{equation}
 \label{1.8}
 Lh=\Delta h+e^u(1-2e^u)h
 \end{equation}
 is an isomorphism from $H^2(\mathbb R^2)$ to $L^2(\mathbb R^2)$ satisfying
 \begin{equation}
 \label{1.9}
 \|Lh\|_{L^2(\mathbb R^2)}\geq C\|h\|_{H^2(\mathbb R^2)}, \quad \forall h\in H^2(\mathbb R^2)
 \end{equation}
 for some constant $C>0$.
}
\par Theorem B tells us that the restriction \eqref{1.5} in Theorem 1.1 is natural as we don't have the uniqueness of topological multivortex solutions of \eqref{1.6} in $\mathbb R^2$. In the following, we will show, in some sense, the uniqueness of topological solutions of \eqref{1.1} is ``equivalent" to the uniqueness of topological solutions of \eqref{1.6}.
%Then the constant $C_1$ in Theorem \ref{thm1.2} is just the constant $C_1$ in Theorem B. Theorem \ref{thm1.2} tells us that when we consider the vortex collapsing case of \eqref{1.1}, then the collapsing velocity is quite important for us to establish the uniqueness of topological solution. However, when the collapsing velocity is not so fast and so slow, we can't prove the uniqueness result. Instead of this, we connect the uniqueness of topological solution on a flat torus $\Omega$ with the uniqueness of topological solution in the whole space $\mathbb R^2$ by additional assumptions.
\begin{theorem}
\label{thm1.3}
Suppose $\psi $ is a topological solution of \eqref{1.6} with its linearized operator $L_\psi$ satisfying \eqref{1.9}. Then there exists a topological solution $u_\varepsilon$ solves
\begin{equation}
\label{1.13}
\Delta u_\varepsilon+\frac 1{\varepsilon^2} e^{u_\varepsilon}(1-e^{u_\varepsilon})=4\pi\sum_{i=1}^l\alpha_i \delta_{\varepsilon p_i},\quad \text{in }\Omega,
\end{equation}
and $\hat u_\varepsilon(x)=u_\varepsilon(\varepsilon x)$ such that
\begin{equation}
\|\hat u_\varepsilon-\psi\|_{L^\infty(B_{d/\varepsilon}(0))}\rightarrow 0, \quad \text{for some constant } d>0 \text{ small}.
\end{equation}
\end{theorem}
A direct consequence of Theorem \ref{thm1.3} is the following corollary.
\begin{corollary}
Suppose, for some configuration $\{p_1,\cdots,p_l\}$, there exist two different topological solutions  $u_1,u_2$ of \eqref{1.6} with their linearized operators $L_{u_1},L_{u_2}$ satisfying \eqref{1.9}. Then  \eqref{1.13} possesses at least two topological solutions for small $\varepsilon$.
\end{corollary}
First, we sketch our proof for Theorem \ref{thm1.2} and Theorem \ref{thm1.3}. Theorem \ref{thm1.2} contains both existence and uniqueness. The existence part will follow from that there exists $\varepsilon_N>0$ depending only on $N$ such that if $0<\varepsilon<\varepsilon_N$, \eqref{1.1} admits a maximal solution. We will revisit the construction of subsolutions in \cite{CaffarelliYang1995}. The uniqueness part of Theorem \ref{thm1.2} is proved by contradiction. Suppose that there exist two sequences of distinct topological solutions $u_{1,\varepsilon}$, $u_{2,\varepsilon}$ of \eqref{1.1}. Then there exists $x_\varepsilon\in\Omega$ such that \begin{equation*}
|u_{1,\varepsilon}(x_\varepsilon)-u_{2,\varepsilon}(x_\varepsilon)|=|u_{1,\varepsilon}-u_{2,\varepsilon}|_{L^\infty(\Omega)}\neq 0,
\end{equation*}
and $x_\varepsilon\rightarrow p$ as $\varepsilon\rightarrow 0$(up to a subsequence). Set
\begin{equation}\label{1.14}
A_\varepsilon=\frac{u_{1,\varepsilon}-u_{2,\varepsilon}}{|u_{1,\varepsilon}-u_{2,\varepsilon}|_{L^\infty(\Omega)}}.
\end{equation}
Then $A_\varepsilon$ satisfies
\begin{equation}
\label{1.10}
\Delta A_\varepsilon+\frac{1}{\varepsilon^2}e^{\tilde u_\varepsilon}(1-2e^{\tilde u_\varepsilon})A_\varepsilon=0,\quad \text{in}\quad \Omega
\end{equation}
where $\tilde u_\varepsilon$ is between $u_{1,\varepsilon}$ and $u_{2,\varepsilon}$. After a suitable scaling at $x_\varepsilon$, \eqref{1.10} converges to a bounded solution $A$ of
\begin{equation}
\label{1.11}
\Delta A+e^{U}(1-2e^{U})A=0,\quad \text{in}\quad \mathbb R^2,
\end{equation}
where $U$ is a topological solution of
\begin{equation}
\label{1.12}
\Delta U+e^{U}(1-e^{U})=4\pi\sum_{i=1}^l \delta_{q_i},\quad \text{in}\quad \mathbb R^2.
\end{equation}
Here $l$ and vortices $\{q_1,\cdots, q_l\}$ are determined by  the rescaling region  and the collapsing velocity of the vortices compared with the coupling constant $\varepsilon$. Then we can apply Theorem B to get contradictions whenever $\{q_1,\cdots,q_l\}$ satisfies the assumption \eqref{1.7}.
%\par  The case  $x_\varepsilon\rightarrow p_1$ needs more consideration where $P=\{p_1,\cdots,p_N\}$ is the limit set of $P_\varepsilon=\{p_1^\varepsilon,\cdots,p_N^\varepsilon\}$ up to a subsequence. Without loss of generality, up to a subsequence, we may assume $p_1^\varepsilon,\cdots,p_l^\varepsilon$ converges to $p_1$ and $p_1^\varepsilon=p_1$ by a shifting. We need to distinguish $p_i^\varepsilon ,p_j^\varepsilon $ by $|p_i-p_j|$
\par The proof of Theorem \ref{thm1.3} depends on the perturbation method which follows from Choe \cite{Choe2005}.  We consider topological solutions of \eqref{1.13} as a perturbation of $\psi_\varepsilon(x)=\psi(\frac x\varepsilon)$. Set $u_\varepsilon(x)=\eta(x)\psi_\varepsilon(x)+\varepsilon^3v_\varepsilon(x)$ and define an operator $G_\varepsilon(v): H^2(\Omega)\rightarrow H^2(\Omega)$. We will prove $G_\varepsilon$ is a well-defined contraction mapping in some suitable space $\mathcal B$.
\par This paper is organized as follows. In Section 2, we will collect some known results and establish some preliminary estimates for the topological solutions which are important to show the convergence of \eqref{1.10} to \eqref{1.11}. In Section 3, we will prove Theorem \ref{thm1.2} i.e. the existence and uniqueness of the topological solution. Section 4 is devoted to the construction of topological solutions which locally converge to a specified topological solution in $\mathbb R^2$.
\section{Preliminaries}
Recall  Green function $G(x,y)$ on $\Omega$,
\begin{equation}
\label{G1}
-\Delta_x G(x,y)=\delta_y-1,\quad x,y\in \Omega, \quad \int_\Omega G(x,y)dx=0.
\end{equation}
We list some properties of $G(x,y)$ as follows:
\begin{itemize}
\item[(a.)] $\forall \phi\in C^2(\Omega)$, $\phi(x)=\int_\Omega \phi(y)dy-\int_\Omega G(x,y)\Delta \phi(y)dy$.
\item[(b.)] $G(x,y)\in C^\infty(\Omega\times\Omega\backslash\{x=y\})$, $G(x,y)=-\frac{1}{2\pi}\ln|x-y|+\gamma(x,y)$ where $\gamma(x,y)$ is the regular part of $G(x,y)$.
\item[(c.)] $|G(x,y)|\leq C(1+|ln|x-y||)$, and $G(x,y)\geq -C_0$, $\forall x,y\in \Omega$ and some constants $C,C_0>0$.
\end{itemize}
We refer the proof of the above properties to \cite{Aubin1982}.
Also remind some facts due to Chan-Fu-Lin \cite{ChanFuLin2002}. Consider
\begin{equation}
\label{CFL1}
\Delta u+ e^u(1-e^u)=0,\quad \text{in}\quad \mathbb R^2,\quad \int_{\mathbb R^2} e^u(1-e^u)dx<\infty.
\end{equation}
Applying the method of moving plane as in Chen-Li \cite{ChenLi1991}, after a translation, all the solutions of \eqref{CFL1} are radially symmetric. Then consider
\begin{equation}
\begin{cases}
\label{CFL2}
u''(r;s)+\frac 1r u'(r;s)+e^{u(r;s)}(1-e^{u(r;s)})=0,\quad r>0,\\
u(0;s)=s,\quad \beta(s)=\int_{\mathbb R^2} e^{u(r;s)}(1-e^{u(r;s)})dx.
\end{cases}
\end{equation}
By \cite{ChanFuLin2002}, $\beta(s)$ is a well-defined continuously differentiable function on $(-\infty,0)$. $\beta(s)$ is monotone increasing such that $\beta(s)\rightarrow +\infty$ as $s\rightarrow 0$, $\beta(s)\rightarrow 8\pi$ as $s\rightarrow -\infty$.
\par  Before the proof of the existence  and uniqueness of topological solutions of  \eqref{1.1}, we need to show the uniform lower bound of the coupling parameter $\varepsilon$ which guarantees the existence of solutions of  \eqref{1.1}. In \cite{CaffarelliYang1995}, the authors  proved the following theorem.

\smallskip
\noindent
{\bf Theorem C.} \cite{CaffarelliYang1995} {\it There is a critical value of $\varepsilon$, say, $\varepsilon(p_1,\cdots,p_N)>0$, satisfying,
\begin{equation*}
\varepsilon(p_1,\cdots,p_N)\leq \frac{1}{\sqrt{16\pi N}}
\end{equation*}
such that, for $0<\varepsilon\leq\varepsilon(p_1,\cdots,p_N)$, \eqref{1.4} has a solution, while for $\varepsilon>\varepsilon(p_1,\cdots,p_N)$, the equation has no solution.
}
\par The proof of Theorem C is based on the construction of sub- and supersolutions. Since the existence of supersolutions always holds true, the delicate part is the construction of subsolution(Lemma 3 in \cite{CaffarelliYang1995}).  However, if we revisit the construction of subsulotions as in Lemma 3 in \cite{CaffarelliYang1995}, we will find $\varepsilon(p_1,\cdots,p_N)$ is independent of the location of vortex.  Although, the construction is similar, we present it here for the  convenience of readers.
\begin{theorem}
\label{thm2.1}
There exists $\varepsilon(N)\geq N^{-cN}$  for some constant $c>0$, such that,  for any configuration $\{p_1,\cdots,p_N\}$,   \eqref{1.1} has a maximal solution provided $\varepsilon<\varepsilon(N)$.
\end{theorem}
\begin{proof}
The proof of Theorem \ref{thm2.1} is also based on the monotone scheme. The only thing we need to take care of is the construction of subsolutions. Set
$$u_0(x)=-4\pi\sum_{i=1}^N G(x,p_i). $$
Then by property $(c)$ of $G(x,y)$, $u_0(x)\leq 4\pi C_0N$. We want to construct a subsolution $v$,
\begin{equation}
\label{sub1}
\Delta v+\frac 1{\varepsilon^2}e^{v+u_0}(1-e^{v+u_0})\geq 4\pi N.
\end{equation}
 Define a smooth function $f_\delta(x)$ as follows:
\begin{equation*}
f_\delta(x)=\begin{cases}1;\quad \forall x\in \cup_{i=1}^{N}B_\delta(p_i);\\
0,\quad \forall x\in \Omega\backslash\cup_{i=1}^{N}B_{2\delta}(p_i)
\end{cases}
\end{equation*}
and $0\leq f_\delta(x)\leq 1$ where $\delta$ is a parameter which will be determined later. By  a direct computation, we get
\begin{equation*}
C(\delta)=\int_\Omega 8\pi N f_\delta\leq 32\pi^2 N^2\delta^2.
\end{equation*}
Denote $g_\delta(x)=8\pi Nf_{\delta}(x)-C(\delta)$. Consider the following equation:
\begin{equation}
\label{2.2}
\Delta w=g_\delta, \quad \text{in}\quad \Omega,\quad  \int_{\Omega} wdx=0.
\end{equation}
Since $\|g_\delta\|_{L^\infty(\Omega)}\leq 8\pi N\|f_\delta\|_{L^\infty(\Omega)}+C(\delta)\leq 8\pi N+32\pi^2 N^2\delta^2$, by  standard $W^{2,p}$ estimates and Sobolev embedding, we get
\begin{equation}
\label{2.3}
\|w\|_{L^\infty(\Omega)}\leq C\|w\|_{W^{2,2}(\Omega)}\leq C\|g_\delta\|_{L^2(\Omega)}\leq C(N+N^2\delta^2).
\end{equation}
Hence we can choose $0<C_1\leq 8\pi(C+C_0)(N+N^2\delta^2)$ such that
\begin{equation}\label{2.4}
w_0=w-C_1, \quad u_0+w_0\leq -4\pi(C+C_0)(N+N^2\delta^2)<\ln \frac 12.
\end{equation}

Let $\delta=\frac{1}{\sqrt{8\pi N}}$. Then $\forall x\in B_\delta(p_i)$, we have
\begin{equation}
\label{2.5}
\begin{split}
\Delta w_0&=g_\delta\geq 4\pi N(2-8\pi N\delta^2)\\
          &\geq 4\pi N\\
          &\geq \frac 1{\varepsilon^2} e^{u_0+w_0}(e^{u_0+w_0}-1)+4\pi N.
\end{split}
\end{equation}
 Set
\begin{equation*}
\begin{split}
&\mu_0=\inf\{e^{u_0+w_0}|x\in \Omega\backslash\cup_{i=1}^N B_\delta(p_i)\},\\
&\mu_1=\sup\{e^{u_0+w_0}|x\in \Omega\backslash\cup_{i=1}^N B_\delta(p_i)\}.
\end{split}
\end{equation*}
$\mu_1<\frac 12$ by  \eqref{2.4}. It remains to estimate $\mu_0$. By the choice of $\delta$, we now have $w_0\geq -C_2 N$. By property $(c)$ of $G(x,y)$, we get
\begin{equation*}
u_0(x)=-4\pi\sum_{i=1}^N G(x,p_i)\geq C_4 N(\ln \delta-1)=-C_5N(\ln N+1),\quad \forall x\in \Omega\backslash \cup_{i=1}^N B_\delta(p_i).
\end{equation*}
The lower bound estimates of $w_0$ and $u_0$ imply that $\mu_0\geq N^{-C_6 N}$. So in order to show
\begin{equation}
\label{2.6}
\Delta w_0=g_\delta\geq \frac 1{\varepsilon^2} e^{u_0+w_0}(e^{u_0+w_0}-1)+4\pi N,\quad x\in \Omega\backslash \cup_{i=1}^N B_\delta(p_i),
\end{equation}
 we only need to guarantee $\frac 1{\varepsilon^2}\mu_0\geq 16\pi N$. This yields $\varepsilon\leq N^{-C_7 N}$. This also implies $w_0$ is a subsolution of \eqref{sub1} for any configuration $\{p_1,\cdots,p_N\}$, $\forall 0<\varepsilon\leq N^{-C_7N}.$ Then the arguments in \cite{CaffarelliYang1995} imply the existence of maximal solutions.
\end{proof}

\begin{lemma}
\label{prop2.1}
Let $u_\varepsilon$ be a sequence of solutions of  \eqref{1.1}. Then up to a subsequence, one of the following holds true:
\begin{itemize}
\item[(i)] $u_\varepsilon\rightarrow -\infty$, a.e. as $\varepsilon\rightarrow 0$.
\item[(ii)] $u_\varepsilon\rightarrow 0$, a.e. as $\varepsilon\rightarrow 0$. Moreover, $u_\varepsilon\rightarrow 0$ in $L^p(\Omega)$, $\forall p\geq 1$.
\end{itemize}
\end{lemma}
\begin{proof}
Set $u_{0,\varepsilon}$ as
\begin{equation*}
\Delta u_{0,\varepsilon}=4\pi\sum_{i=1}^N\delta_{p_i^\varepsilon}-4\pi N,\quad \int_{\Omega} u_{0,\varepsilon}dx=0.
\end{equation*}
Let $d_\varepsilon=\int_{\Omega}u_\varepsilon dx$ and $u_\varepsilon=u_{0,\varepsilon}+w_\varepsilon+d_\varepsilon$. Then $w_\varepsilon$ satisfies
\begin{equation*}
\Delta w_\varepsilon+\frac 1{\varepsilon^2}e^{u_\varepsilon}(1-e^{u_\varepsilon})=4\pi N.
\end{equation*}
Claim: $\exists C_q>0$ such that $\|\nabla w_\varepsilon\|_{L^q(\Omega)}\leq C_q$, $\forall q\in (1,2)$.
\\ Let $q'=\frac{q}{q-1}>2$. Then
\begin{equation*}
\|\nabla w_\varepsilon\|_{L^q(\Omega)}\leq \sup\{\int_\Omega \nabla w_\varepsilon\nabla \varphi |\varphi\in W^{1,q'}(\Omega),\int_\Omega \varphi dx=0,\|\nabla \varphi\|_{L^{q'}(\Omega)}\leq 1\}.
\end{equation*}
By Poincar$\acute{e}$ inequality and Sobolev embedding theorem, we see that
\begin{equation*}
\|\varphi\|_{L^\infty(\Omega)}\leq C(\|\varphi\|_{L^{q'}(\Omega)}+\|\nabla \varphi\|_{L^{q'}(\Omega)})\leq C\|\nabla \varphi\|_{L^{q'}(\Omega)}.
\end{equation*}
In view of the equation $w_\varepsilon$ satisfying, one gets
\begin{equation*}
\begin{split}
\left|\int_\Omega \nabla w_\varepsilon\cdot\nabla \varphi dx\right|=\left|\int_\Omega\Delta w_\varepsilon \varphi dx\right|\leq \|\varphi\|_{L^\infty(\Omega)}\left|\int_\Omega \frac{1}{\varepsilon^2}|e^{u_\varepsilon}(1-e^{u_\varepsilon})|dx+4\pi N\right|\leq C
\end{split}
\end{equation*}
which implies $\|\nabla w_\varepsilon\|_{L^q(\Omega)}\leq C_q$. By Poincar$\acute{e}$ inequality, we also have
\begin{equation*}
\|w_\varepsilon\|_{L^q(\Omega)}\leq C\|\nabla w_\varepsilon\|_{L^q(\Omega)}.
\end{equation*}
Therefore, up to a subsequence, there exists $w\in W^{1,q}(\Omega)$ and $\forall p\geq 1$ such that
\begin{equation*}
\begin{split}
w_\varepsilon\rightharpoonup w\quad \text{in}\quad W^{1,q}(\Omega),\quad w_\varepsilon\rightarrow w\quad \text{in} \quad L^p(\Omega),\quad w_\varepsilon\rightarrow w\quad \text{a.e..}
\end{split}
\end{equation*}
It also follows from the definition of $u_{0,\varepsilon}$ that $u_{0,\varepsilon}\rightarrow u_0$ in $L^p(\Omega)$ for all $p\geq 1$.
We consider the following two possible cases.
\begin{itemize}
\item[1.] $\lim\sup_{\varepsilon\rightarrow 0}\frac{e^{d_\varepsilon}}{\varepsilon^2}\leq C$ for some constant $C$.
\item[2.] $\lim\sup_{\varepsilon\rightarrow 0}\frac{e^{d_\varepsilon}}{\varepsilon^2}=+\infty$.
\end{itemize}
If Case 1 happens, then
\begin{equation*}
e^{u_\varepsilon}=e^{w_\varepsilon+u_{0,\varepsilon}+d_\varepsilon}\leq C\varepsilon^2 e^{w_\varepsilon+u_{0,\varepsilon}}\rightarrow 0
\end{equation*}
which implies $u_\varepsilon\rightarrow -\infty$ a.e..
\\ Now we assume Case 2 happens. Since $u_\varepsilon<0$ and $\int_\Omega u_{0,\varepsilon}dx=0$, we see that $0\leq e^{d_\varepsilon}\leq 1$. This means we can find $A\geq 0$ such that
\begin{equation*}
\lim_{\varepsilon\rightarrow 0}\sup e^{d_\varepsilon}=A.
\end{equation*}
By Fatou's lemma, we get
\begin{equation*}
4\pi N\varepsilon^2e^{-d_\e}=\int_\Omega e^{w_\varepsilon+u_{0,\e}}(1-e^{u_\varepsilon})\geq \int_\Omega e^{w+u_0}(1-Ae^{w+u_0})dx.
\end{equation*}
 Since $\int_{\Omega }w+u_0dx=0$, we get $A\equiv 1.$ This ends the proof of Lemma \ref{prop2.1}.
\end{proof}

 Since our arguments always are along a subsequence, without loss of generality, in the arguments below, we always assume
\begin{equation*}
P_n=\{p_1^{\varepsilon_n},\cdots,p_N^{\varepsilon_n}\}=\{p_{1,n},\cdots,p_{N,n}\}\rightarrow \{p_1,\cdots,p_N\}=P,
\end{equation*}
$u_n\rightarrow 0$ a.e. solves
\begin{equation}
\label{new1}
\Delta u_n+\frac 1{\varepsilon_n^2} e^{u_n}(1-e^{u_n})=4\pi\sum_{i=1}^N \delta_{p_{i,n}},\quad \text{in}\quad \Omega.
\end{equation}
Set $d=\frac 14\inf_{p_i\neq p_j} |p_i-p_j|$. Then we have the following lemma.

\begin{lemma}
\label{cor2.1}
Suppose $u_n\rightarrow 0$ a.e. solves \eqref{new1}. Then for any $R_n\rightarrow +\infty$, $R_n\varepsilon_n\leq d$, we have
\begin{equation*}
\|u_n\|_{L^\infty(\Omega\backslash\cup_{i=1}^N B_{R_n\varepsilon_n}(p_{i,n}))}\rightarrow 0.
\end{equation*}
\end{lemma}
\begin{proof}
 First, we need to show for any compact set $K\subset \Omega\backslash P$,
 \begin{equation*}
 \|u_n\|_{L^\infty(K)}\rightarrow 0.
 \end{equation*}
 Set $\delta=\frac 14 d(K,P)$. Since $u_n<0$ on $\Omega$, we note that $u_n$ is subharmonic in $B_\delta(x)$ for any $x\in K$. By using the mean value theorem and Lemma \ref{prop2.1}, we get
 \begin{equation*}
 0\leq -u_n(x)\leq \frac 1{|B_\delta|}\|u_n\|_{L^1(\Omega)}\rightarrow 0, \quad \text{as}\quad n\rightarrow \infty, \quad \forall x\in K.
 \end{equation*}
Next, we  classify the points $\{p_{1,n},\cdots,p_{N,n}\}$ according to their asymptotic behavior.
Define $A_{k,n}$ as
\begin{equation*}
A_{k,n}=\{p_{i,n}|\lim_{n\rightarrow \infty}\frac{|p_{i,n}-p_{k,n}|}{\varepsilon_n}<+\infty\}.
\end{equation*}
It's obvious that for any $i,j$, either $A_{i,n}=A_{j,n}$ or $A_{i,n}\cap A_{j,n}=\emptyset$ along a suitable subsequence. Without loss of generality, we can take $A_{i,n}$, $i=1,\cdots,l$
such that
\begin{equation*}
A_{i,n}\cap A_{j,n}=\emptyset,\quad i\neq j,\quad \cup_{i=1}^l A_{i,n}=\{p_{1,n},\cdots,p_{N,n}\}.
\end{equation*}
Set
$$r_n=\frac 14\min(\inf_{i\neq j}\frac{d(A_{i,n},A_{j,n})}{\varepsilon_n},R_n)\rightarrow +\infty.$$
So it's enough to prove that
\begin{equation*}
\|u_n\|_{L^\infty(\Omega\backslash\cup_{i=1}^lB_{r_n\varepsilon_n}(q_{i,n}))}\rightarrow 0, \quad q_{i,n}\in A_{i,n}, \quad i=1,\cdots,l.
\end{equation*}
It follows from our choice that $\cup_{i=1}^l B_{r_n\varepsilon_n}(q_{i,n})$ covers the set $\{p_{1,n},\cdots,p_{N,n}\}$ and $B_{r_n\varepsilon_n}(q_{i,n})$ $\cap B_{r_n\varepsilon_n}(q_{j,n})=\emptyset$ for $i\neq j$. Hence, the set $B_d(0)\backslash \cup_{i=1}^l B_{r_n\varepsilon_n}(q_{i,n})$ is pathwise connected. Now suppose $u_n(x_n)\leq -c_0<0$, $x_n\in \Omega\backslash\cup_{i=1}^lB_{r_n\varepsilon_n}(q_{i,n})$. Then by our first step in the proof, without loss of generality, we may assume $x_n\rightarrow 0\in P$. Choosing smooth curves $\gamma_n\subset B_d(0)\backslash \cup_{i=1}^l B_{r_n\varepsilon_n}(q_{i,n})$ joining $x_n$ and $y_0\in \partial B_d(0)$, by intermediate value theorem, we get $z_n\in \gamma_n$, $u_n(z_n)=s_0<0$ such that $\beta(s_0)>4\pi N$ with $\beta(s)$ defined in \eqref{CFL2} since it's already known $u_n(y_0)\rightarrow 0$. Set $\hat u_n(x)=u_n(\varepsilon_n x+z_n)$. Then $\hat u_n(x)$ solves
\begin{equation*}
\begin{cases}
\Delta \hat u_n+e^{\hat u_n}(1-e^{\hat u_n})=0,\quad \text{in}\quad B_{r_n/2}(0),\\
\hat u_n(x)<0,\hat u_n(0)=s_0<0,\int_{B_{r_n/2}(0)} e^{\hat u_n}(1-e^{\hat u_n})\leq 4\pi N.
\end{cases}
\end{equation*}
As $\left|\frac{e^{\hat u_n}(1-e^{\hat u_n})}{\hat u_n}\right|\leq C$ uniformly, by Harnack inequality(Theorem 8.20 in \cite{GilbargTrudinger2001}), we have $\hat u_n$ is locally uniformly bounded. Applying standard $W^{2,p}$ and Schauder's estimates, one gets
\begin{equation*}
\hat u_n\rightarrow \hat u,\quad \text{in} \quad C^2_{loc}(\mathbb R^2)
\end{equation*}
and $\hat u$ solves \eqref{CFL1} with $\hat u(0)=s_0$. Then
\begin{equation*}
4\pi N\geq \int_{\mathbb R^2} e^{\hat u}(1-e^{\hat u})dx\geq \beta(s_0)>4\pi N.
\end{equation*}
This yields a contradiction and proves present lemma.
\end{proof}

\begin{lemma}
\label{lem2.2}
Suppose $u_n$ is a topological solution of  \eqref{1.1}. Then up to a subsequence,
 for any $R_n\rightarrow +\infty$, $R_n\varepsilon_n\leq d$, we have
\begin{equation*}
\varepsilon_n^k|D^k u_n|_{L^\infty(\Omega\backslash\cup_{i=1}^N B_{R_n\varepsilon_n}(p_{i,n}))}\rightarrow 0,
\end{equation*}
for any $k\in \mathbb N$ and faster that any other power of $\frac{1}{R_n}$.
\end{lemma}
\begin{proof}
Set $\Omega_n=\Omega\backslash\cup_{i=1}^N B_{R_n\varepsilon_n}(p_{i,n})$, $\Omega_n'=\Omega\backslash\cup_{i=1}^N B_{2R_n\varepsilon_n}(p_{i,n})$. By Lemma \ref{cor2.1}, we have
\begin{equation}
\label{2.7}
\begin{split}
\int_{\Omega_n}|u_n|dx&\leq (1+|u_{n}|_{L^\infty(\Omega_n)})\int_{\Omega_n}\frac{|u_n|}{1+|u_n|}\\
&\leq (1+|u_n|_{L^\infty(\Omega_n)})\int_{\Omega_n}(1-e^{u_n})\\
&\leq (1+|u_n|_{L^\infty(\Omega_n)})e^{|u_n|_{L^\infty(\Omega_n)}}\int_{\Omega_n} e^{u_n}(1-e^{u_n})\leq C\varepsilon_n^2.
\end{split}
\end{equation}
In getting \eqref{2.7}, we have used $|1-e^t|\geq \frac{|t|}{1+|t|}$, $\forall t\in\mathbb R$.
Then for any $x\in\Omega_n'$, we note that $u_n$ is subharmonic in $B_{R_n\varepsilon_n}(x)$. By mean value theorem and \eqref{2.7}, we get
\begin{equation*}
|u_n(x)|\leq \frac{C}{R_n^2\varepsilon_n^2}\int_{B_{R_n\varepsilon_n}(x)}|u_n(y)|dy\leq \frac{C}{R_n^2},\quad \forall x\in \Omega_n'.
\end{equation*}
Define $0\leq \varphi_n\in C^{\infty}(\Omega)$ such that
\begin{equation*}
\varphi_n\equiv 0,\quad \text{in}\quad \cup_{i=1}^N B_{R_n\varepsilon_n}(p_{i,n}),\quad \varphi\equiv 1,\quad \text{in}\quad \Omega_n',\quad |D^k \varphi_n|\leq C_k(R_n\varepsilon_n)^{-k},\quad k=1,2,\cdots.
\end{equation*}
Then
\begin{equation}
\label{2.8}
\begin{split}
\int_{\Omega_n'}\frac 1{\varepsilon_n^2}e^{u_n}(1-e^{u_n})&\leq \frac 1{\varepsilon_n^2}\int_{\Omega_n}e^{u_n}(1-e^{u_n})\varphi_n\\
&=-\int_{\Omega} u_n\Delta\varphi_ndx\leq \frac{C}{R_n^2\varepsilon_n^2}\int_{\Omega_n}|u_n|\leq \frac C{R_n^2}.
\end{split}
\end{equation}
Repeating the procedure of  \eqref{2.7} and \eqref{2.8}, we have
\begin{equation*}
\int_{\Omega_n}|u_n(x)|dx\leq \frac{C_m\varepsilon_n^2}{R_n^{2m}},\quad \int_{\Omega_n}\frac{1}{\varepsilon_n^2}e^{u_n}(1-e^{u_n})\leq \frac{C_m}{R_n^{2m}}.
\end{equation*}
This implies
\begin{equation*}
\|u_n\|_{L^\infty(\Omega_n)}\leq \frac{C_m}{R_n^{2m}},\quad \|\frac 1{\varepsilon_n^2}e^{u_n}(1-e^{u_n})\|_{L^\infty(\Omega_n)}\leq \frac{C_m}{R_n^{2m}\varepsilon_n^2}.
\end{equation*}
Set $\hat u_n(x)=u_n(\varepsilon_n x)$. Then $\hat u_n$ solves
\begin{equation*}
\Delta \hat u_n+e^{\hat u_n}(1-e^{\hat u_n})=0, \quad \text{in}\quad \frac{\Omega_n}{\varepsilon_n}.
\end{equation*}
Since $|\hat u_n|\leq C_m/R_n^{2m}$ in $\Omega_n/\varepsilon_n$, by standard $W^{2,p}-$estimates and Schauder estimates, we get
\begin{equation*}
|D^k\hat u_n|_{L^\infty(\Omega_n/\varepsilon_n)}\leq \frac {C_{k,m}}{R_n^{2m}},\quad \forall k\geq 0.
\end{equation*}
Scaling back to $u_n(x)$,  one get
$$|D^k u_n|_{L^\infty(\Omega\backslash\cup_{i=1}^N B_{R_n\varepsilon_n}(p_{i,n}))}\leq \frac{C_m}{R_n^{2m}\varepsilon_n^k}.$$
\end{proof}
If no confuse occurs, in the remaining part of this paper, $A_{k,n}$, $q_{k,n}$, $k=1,\cdots,l$, $r_n$  always mean the terminologies defined in Lemma \ref{cor2.1}. Suppose $p_{1,n},\cdots,p_{l_{1,1},n}\in A_{1,n}$, $l_{1,1}$ is the number of elements in $A_{1,n}$ and $p_{1,n}=0$ after a shift of coordinates. Then set $v_n=u_n(x)-2\sum_{i=1}^{l_{1,1}}\ln |x-p_{i,n}|$. We have the following important a priori estimates.
\begin{lemma}
\label{lem2.3}
Suppose $u_n$ is a topological solution of \eqref{new1}. Then
\begin{equation*}
\frac{1}{\varepsilon_n^2}e^{u_n}(1-e^{u_n})\rightarrow 4\pi\sum_{i\in I}l_i\delta_{p_i},
\end{equation*}
weakly in the sense of measure in $\Omega$, where $I\subset \{1,\cdots,N\}$ is a set of indices identifying all distinct vortices in $\{p_1,\cdots,p_N\}$, $l_i\in \mathbb N$ is the multiplicity of $p_i$, $i\in I$. And for any $\tilde r_n\leq r_n$, $\tilde r_n\rightarrow +\infty$, we have
\begin{equation*}
\int_{B_{\tilde r_n\varepsilon_n}(0)}\frac 1{\varepsilon_n^2}(1-e^{u_n})^2=4\pi l_{1,1}^2+4\pi\sum_{i=1}^{l_{1,1}}p_{i,n}\cdot\nabla v_n(p_{i,n})+o(1),\  \int_{\Omega\backslash \cup_{i=1}^{l} B_{\tilde r_n\varepsilon_n}(q_{i,n})}\frac 1{\varepsilon_n^2}(1-e^{u_n})^2=o(1).
\end{equation*}
\end{lemma}
\begin{proof}
The first part is easy. Since for any $\varphi\in C^\infty(\Omega)$, we have
\begin{equation*}
\begin{split}
&\left|\frac 1{\varepsilon_n^2}\int_{\Omega}e^{u_n}(1-e^{u_n})\varphi-4\pi\sum_{i=1}^N\varphi(p_{i,n})\right|=\left|-\int_\Omega \Delta u_n\varphi\right|=\left|-\int_\Omega\Delta \varphi u_n\right|\\
&\leq \|\Delta \varphi\|_{L^\infty(\Omega)}\|u_n\|_{L^1(\Omega)}\rightarrow 0,\quad \text{as} \quad \varepsilon_n\rightarrow 0
\end{split}
\end{equation*}
and this proves the first part if we notice $P_n\rightarrow P$.
\par  Denote $\displaystyle f_n=2\sum_{i=1}^{l_{1,1}}\ln|x-p_{i,n}|$. And
 $v_n$ solves
\begin{equation}
\Delta v_n+\frac 1{\varepsilon_n^2}e^{u_n}(1-e^{u_n})=0,\quad \text{in}\quad B_{\tilde r_n\varepsilon_n}(0).
\end{equation}
By Lemma \ref{lem2.2} and the property of $p_{i,n}$, $i=1,\cdots,l_{1,1}$, we have
\begin{equation}
|\nabla v_n+\frac{2l_{1,1}\nu}{\tilde r_n\varepsilon_n}|=O(\frac{1}{\tilde r_n^2\varepsilon_n}),\quad \text{on}\quad \partial B_{\tilde r_n\varepsilon_n}(0)
\end{equation}
where $\nu$ is the outward normal of $\partial B_{\tilde r_n\varepsilon_n}(0)$.
In view of
\begin{equation}
\label{2.9}
\begin{split}
\int_{B_{\tilde r_n\varepsilon_n}(0)}(x\cdot\nabla v_n)\Delta v_n&=\int_{\partial B_{\tilde r_n\varepsilon_n}(0)}(x\cdot\nabla v_n)(\nabla v_n\cdot\nu)-\frac 12\int_{\partial B_{\tilde r_n\varepsilon_n}(0)}(x\cdot\nu)|\nabla v_n|^2\\
&=\frac 12\int_{\partial B_{\tilde r_n\varepsilon_n}(0)}\left(O\left(\frac 1{\tilde r_n}\right)-2l_{1,1}\right)\left(O\left(\frac{1}{\tilde r_n^2\varepsilon_n}\right)-\frac {2l_{1,1}}{\tilde r_n\varepsilon_n}\right)\\
&=4\pi l_{1,1}^2+O\left(\frac 1{\tilde r_n}\right).
\end{split}
\end{equation}
And
\begin{equation}\label{2.10}
\begin{split}
&\int_{B_{\tilde r_n\varepsilon_n}(0)}(x\cdot\nabla v_n)\frac{1}{\varepsilon_n^2}e^{u_n}(1-e^{u_n})\\
=&\int_{B_{\tilde r_n\varepsilon_n}(0)}(x\cdot\nabla u_n)\frac{1}{\varepsilon_n^2}e^{u_n}(1-e^{u_n})-(x\cdot\nabla f_n)(-\Delta v_n)\\
=&\int_{B_{\tilde r_n\varepsilon_n}(0)}\frac 1{\varepsilon_n^2}(1-e^{u_n})^2-\frac 1{2\varepsilon_n^2}\int_{\partial B_{\tilde r_n\varepsilon_n}(0)}(x\cdot\nu)(1-e^{u_n})^2+\int_{B_{\tilde r_n\varepsilon_n}(0)}(x\cdot\nabla f_n)\Delta v_n.
\end{split}
\end{equation}
And

\begin{eqnarray}
&&\int_{B_{\tilde r_n\varepsilon_n}(0)}(x\cdot\nabla f_n)\Delta v_n\nonumber\\
&=&-8\pi l_{1,1}^2+2\lim_{\delta\rightarrow 0}\sum_{i=1}^{l_{1,1}}\int_{B_{\tilde r_n\varepsilon_n}(0)\backslash B_\delta(p_{i,n})}\nabla\cdot[p_{i,n}\cdot(\ln|x-p_{i,n}|-\ln(\tilde r_n\varepsilon_n))]\Delta v_n+o(1)\nonumber\\
&=&-8\pi l_{1,1}^2-4\pi\sum_{i=1}^{l_{1,1}}p_{i,n}\cdot\nabla v_n(p_{i,n})+2\pi\sum_{i=1}^{l_{1,1}}\int_{\partial B_{\tilde r_n\varepsilon_n(0)}}(p_{i,n}\cdot\nu)(\ln|x-p_{i,n}|-\ln(\tilde r_n\varepsilon_n))\Delta v_n\nonumber\\
&&-2\pi\sum_{i=1}^{l_{1,1}}\int_{\partial B_{\tilde r_n\varepsilon_n(0)}}\nabla (p_{i,n}\cdot\nabla v_n)\cdot\nu
(\ln|x-p_{i,n}|-\ln(\tilde r_n\varepsilon_n))+o(1)\label{2.11}\\
&=&-8\pi l_{1,1}^2-4\pi\sum_{i=1}^{l_{1,1}}p_{i,n}\cdot\nabla v_n(p_{i,n})+o(1)\nonumber.
\end{eqnarray}
In order to get  \eqref{2.9}-\eqref{2.11}, we have used the estimates in Lemma \ref{lem2.2} and $\ln|x-p_{in}|-\ln(\tilde r_n\varepsilon_n)\rightarrow 0$ on $\partial B_{\tilde r_n\varepsilon_n}(0)$. Combining \eqref{2.9}-\eqref{2.11}, we get
\begin{equation}
\label{2.12}
\frac 1{\varepsilon_n^2}\int_{B_{\tilde r_n\varepsilon_n}(0)}(1-e^{u_n})^2=4\pi l_{1,1}^2+4\pi\sum_{i=1}^{l_{1,1}}p_{i,n}\cdot\nabla v_n(p_{i,n})+o(1).
\end{equation}
Now we want to show $p_{i,n}\cdot \nabla v_n(p_{i,n})$ is uniformly bounded for $i=1,\cdots,l_{1,1}$. By Green's representation formula, one gets
\begin{equation}
u_n-u_{0,n}=\int_\Omega u_n(y)dy+\int_\Omega G(x,y)\frac 1{\varepsilon_n^2}e^{u_n}(1-e^{u_n})dy.
\end{equation}
Since $G(x,y)=-\frac 1{2\pi}\ln|x-y|+\gamma(x,y)$ where $\gamma(x,y)$ is the regular part, one can expect that for $x\in B_{R\varepsilon_n}(0)$ for some fixed $R$ large enough,
\begin{eqnarray}
\label{2.13}
&&|\nabla u_n-2\sum_{i=1}^{l_{1,1}}\frac{x-p_{i,n}}{|x-p_{i,n}|^2}|\nonumber\\
&\leq& C+\int_{\Omega}\frac 1{|x-y|}\frac 1{\varepsilon_n^2} e^{u_n}(1-e^{u_n})dy+2\sum_{i=l_{1,1}+1}^{l_1}\frac{x-p_{i,n}}{|x-p_{i,n}|^2}\\
&\leq& C+\int_{B_{\varepsilon_n}(x)}\frac 1{|x-y|}\frac 1{\varepsilon_n^2} e^{u_n}(1-e^{u_n})dy+\int_{\Omega\backslash B_{\varepsilon_n}(x)}\frac 1{|x-y|}\frac 1{\varepsilon_n^2} e^{u_n}(1-e^{u_n})dy+o\left(\frac 1{\varepsilon_n}\right)\nonumber\\
&\leq& C+\frac C{\varepsilon_n}\nonumber.
\end{eqnarray}
By  \eqref{2.13} and the uniform bound of  $\frac{|p_{i,n}|}{\varepsilon_n}$, $i=1,\cdots,l_{1,1}$,  we can get $p_{i,n}\cdot\nabla v_n(p_{i,n})$ is uniformly bounded for $i=1,\cdots,l_{1,1}$.

From $|\frac 1{\varepsilon_n^2}e^{u_n}(1-e^{u_n})|_{L^1(\Omega)}=4\pi N$ and $u_n<0$, we see that for $\Omega_n= \Omega\backslash\cup_{i=1}^l B_{\tilde r_n\varepsilon_n}(q_{i,n})$
\begin{equation*}
\frac{1}{\varepsilon_n^2}\int_{\Omega_n}(1-e^{u_n})^2\leq\sup_{\Omega_n}(1-e^{u_n})\inf_{\Omega_n}e^{u_n}\int_{\Omega_n}\frac{1}{\varepsilon_n^2}e^{u_n}(1-e^{u_n})\rightarrow 0.
\end{equation*}
\end{proof}
Denote
 \begin{equation*}
 \begin{split}
 \hat u_n(x)=u_n(\varepsilon_n x), \quad \frac{p_{i,n}}{\varepsilon_n}\rightarrow \hat p_i,\quad i=1,\cdots,l_{1,1}.
\end{split}
\end{equation*}
\begin{lemma}
\label{lem2.4}
Suppose  $\hat u(x)$ is the unique topological solution of
\begin{equation*}
\begin{cases}
\displaystyle\Delta \hat u+e^{\hat u}(1-e^{\hat u})=4\pi\sum_{i=1}^{l_{1,1}}\delta_{\hat p_i},\quad \text{in}\quad \mathbb R^2,\\
\displaystyle\hat u<0,\sup_{\mathbb R^2\backslash \cup_{i=1}^{l_{1,1}}B_1(\hat p_i)}|\nabla \hat u|<+\infty,\quad \int_{\mathbb R^2}(1-e^{\hat u})dx<+\infty
\end{cases}
\end{equation*}
where $\hat p_i$ satisfies the assumptions in Theorem B. Set
\begin{equation*}
\hat v_n(x)=\hat u_n(x)-\sum_{i=1}^{l_{1,1}}\ln\frac{| x-\frac{p_{i,n}}{\varepsilon_n}|^2}{1+| x-\frac{p_{i,n}}{\varepsilon_n}|^2}=\hat u_n-h_n,\quad \hat v(x)=\hat u(x)-\sum_{i=1}^{l_{1,1}}\ln\frac{|x-\hat p_i|^2}{1+|x-\hat p_i|^2}=\hat u-h.
\end{equation*}
Then $\displaystyle \lim_{\varepsilon\rightarrow 0}\sup_{B_{ r_n}(0)}|\hat v_n-\hat v|=0$.
\end{lemma}
\begin{proof}
Claim: $\displaystyle \sup_{B_{r_n}(0)\backslash\cup_{i=1}^{l_{1,1}}B_1(\hat p_i)}|\hat u_n|\leq C.$
\\ Suppose the claim isn't true. We can pick up $z_n\in B_{r_n}(0)\backslash\cup_{i=1}^{l_{1,1}}B_1(\hat p_i)$ such that $\hat u_n(z_n)\rightarrow -\infty$. By Lemma \ref{lem2.2}, $z_n$ must be uniformly bounded. Otherwise, one can define $\tilde r_n=\frac{|z_n|}{4}\rightarrow +\infty$ and $\Omega_n=\Omega\backslash\cup_{i=1}^l B_{\tilde r_n\varepsilon_n}(q_{i,n})$. Then Lemma \ref{lem2.2} tells us that $\|u_n\|_{L^\infty(\Omega_n)}\rightarrow 0$ which contradicts to $u_n(z_n\varepsilon_n)=\hat u_n(z_n)\rightarrow -\infty$ as $z_n\varepsilon_n \in \Omega_n$. Because $\hat u_n$ solves
\begin{equation*}
\Delta \hat u_n+e^{\hat u_n}(1-e^{\hat u_n})=0,\quad \text{in}\quad B_{r_n}(0)\backslash\cup_{i=1}^{l_{1,1}}B_1(\hat p_i),\quad \hat u_n<0
\end{equation*}
and $\left|\frac{e^{\hat u_n}(1-e^{\hat u_n})}{\hat u_n}\right|\in L^\infty$, by Harnack inequality, we have
\begin{equation*}
\hat u_n\rightarrow -\infty,\quad \text{locally in any compact set }K\subset B_{r_n}(0)\backslash\cup_{i=1}^{l_{1,1}}B_1(\hat p_i).
\end{equation*}
This implies that
\begin{equation*}
\int_{B_R(0)}(1-e^{\hat u_n})^2\geq \frac{1}{4}\pi R^2\rightarrow +\infty,\quad \text{as} \quad R\rightarrow +\infty
\end{equation*}
which contradicts to
\begin{equation*}
\int_{B_R(0)}(1-e^{\hat u_n})^2\leq \frac 1{\varepsilon_n^2}\int_{B_{r_n\varepsilon_n}(0)}(1-e^{u_n})^2\leq C.
\end{equation*}
This proves our claim. Since $\hat v_n$ solves
\begin{equation*}
\Delta\hat v_n +e^{\hat u_n}(1-e^{\hat u_n})=\sum_{i=1}^{l_{1,1}}\frac{4}{(1+|x-\frac{p_{i,n}}{\varepsilon_n}|^2)^2}=g_n,\quad \text{in}\quad B_{r_n}(0).
\end{equation*}
By the claim, we have
\begin{equation*}
\|\hat v_n\|_{L^\infty(\partial B_R)}\leq C_R,\quad \text{for}\quad R\geq \max_{1\leq i\leq l_{1,1}}|\hat p_i|+1.
\end{equation*}
Applying maximum principle to $\Delta\hat  v_n\pm (4l_{1,1}+1)\gtrless 0$, one gets $\|\hat v_n\|_{L^\infty(B_R)}\leq C_R$. From the standard $W^{2,p}$ estimates, we have
$\hat v_n\rightarrow \hat v$ in $W^{2,p}_{loc}(\mathbb R^2)$ for $p>1$. By Sobolev embedding theorem, we have $\nabla \hat v_n(\frac{p_{i,n}}{\varepsilon_n})\rightarrow \nabla\hat v(\hat p_i)$. This implies
\begin{equation*}
\int_{B_{r_n}(0)}(1-e^{\hat u_n})^2=4\pi l_{1,1}^2+4\pi\sum_{i=1}^{l_{1,1}}\hat p_i\cdot\nabla\hat v(\hat p_i)+o(1).
\end{equation*}
Then by Fatou's Lemma
\begin{equation*}
\int_{B_{r_n}(0)}(e^{\hat u_n}-e^{\hat u})^2=\int_{B_{r_n}(0)}(e^{\hat u_n}-1)^2+(e^{\hat u}-1)^2-2\int_{B_{r_n}(0)}(1-e^{\hat u_n})(1-e^{\hat u})\leq o(1).
\end{equation*}
This again implies
\begin{equation*}
o(1)=\int_{B_{r_n}(0)\backslash B_{R}(0)} (e^{\hat u_n}-e^{\hat u})^2\geq \min_{B_{r_n}(0)\backslash B_{R}(0)}(e^{2\hat u_n},e^{2\hat u})\int_{B_{r_n}(0)\backslash B_{R}(0)} (\hat u_n-\hat u)^2,
\end{equation*}
or $\int_{B_{r_n}(0)\backslash B_{R}(0)} (\hat u_n-\hat u)^2=o(1)$.
Direct computation yields that
\begin{equation*}
|h_n-h|_{L^2(B_{r_n}(0)\backslash B_{R}(0))}+|g_n-g|_{L^2(B_{r_n}(0))}=o(1),\quad \text{as}\quad n\rightarrow \infty.
\end{equation*}
By the local convergence of $\hat v_n\rightarrow \hat v$ and standard $W^{2,p}$ estimates, we have
$$\lim_{\varepsilon_n\rightarrow 0}\sup_{B_{r_n}(0)} |\hat v_n-\hat v|\rightarrow 0.$$
\end{proof}

\section{The existence and uniqueness of topological solutions}
 \textbf{Proof of Theorem \ref{thm1.2}}. By Theorem \ref{thm2.1}, for any configuration $\{p_1,\cdots,p_N\}$, $\exists \varepsilon_N>0$,  $\forall 0<\varepsilon\leq \varepsilon_N$,  \eqref{1.4} admits a maximal solution. Suppose   $u_n$ is a sequence of  maximal solutions  of \eqref{1.1} with configuration $\{p_{1,n},\cdots,p_{N,n}\}$ and coupling parameter $\varepsilon_n\rightarrow 0$. By Lemma \ref{prop2.1}, we have either $u_n\rightarrow 0$ a.e. or $u_n\rightarrow-\infty$ a.e.. Set $v_n=u_n-u_{0,n}$. Since $w_0$ constructed in Theorem \ref{thm2.1} is a subsolution of
 \begin{equation*}
 \Delta v_n+\frac{1}{\varepsilon_n^2}e^{v_n+u_{0,n}}(1-e^{v_n+u_{0,n}})=4\pi N,
 \end{equation*}
 by the monotone decreasing property of maximal solutions with respect to $\varepsilon$, we have $v_n\geq w_0$ which is uniformly bounded from below. This implies $u_n=v_n+u_{0,n}\rightarrow 0$ a.e. and proves the existence part.

Now suppose there are two different topological solutions of  \eqref{1.1} $u_{1,n}$, $u_{2,n}$. Set
$$\phi_n(x)=\frac{(u_{1,n}-u_{2,n})(x)}{|u_{1,n}-u_{2,n}|_{L^\infty(\Omega)}},\quad |\phi_n(x_n)|=\|\phi_n\|_{L^\infty(\Omega)}=1. $$
We consider the following two cases.
\begin{itemize}
\item[(I)] $\tilde r_n=\frac 14\min(\frac{|x_n-p_{1,n}|}{\varepsilon_n},\cdots,\frac{|x_n-p_{N,n}|}{\varepsilon_n})\rightarrow +\infty$. By Lemma \ref{lem2.2},  one has
    \begin{equation*}
    \|u_{1,n}\|_{L^\infty(B_{\tilde r_n\varepsilon_n}(x_n))}, \|u_{2,n}\|_{L^\infty(B_{\tilde r_n\varepsilon_n}(x_n))}\rightarrow 0 .
    \end{equation*}
    Set $\hat \phi_n(x)=\phi_n(\varepsilon_n x+x_n)$, $\hat u_{i,n}(x)=u_{i,n}(\varepsilon_n x+x_n)$, $i=1,2$. Then $\hat \phi_n$ satisfies
    \begin{equation}
    \Delta \hat\phi_n+e^{\hat u_n}(1-2e^{\hat u_n})\hat \phi_n=0,\quad \text{in}\quad B_{\tilde r_n}(0),
    \end{equation}
    where $\hat u_n$ is between $\hat u_{1,n}$, $\hat u_{2,n}$.
      From standard $W^{2,p}$ estimates and Schauder's estimates, we obtain a subsequence $\hat \phi_n\rightarrow \hat \phi$ in $C^2_{loc}(\mathbb R^2)$ with $\hat\phi$ satisfying
    \begin{equation*}
    \begin{split}
    &\Delta \hat \phi-\hat \phi=0,\quad\text{in}\quad \mathbb R^2,\\
    &|\hat \phi(0)|=\|\hat \phi\|_{L^\infty}=1.
    \end{split}
    \end{equation*}
\item[(II)] $\tilde r_n\leq C<+\infty$. Without loss of generality, we may assume $p_{1,n}=0$ and $\frac{|x_n|}{\varepsilon_n}\leq C$. Set $\hat \phi_n(x)=\phi_n(\varepsilon_n x+x_n).$ Then as in Case (I), by Lemma \ref{lem2.4} and Theorem B, we obtain a subsequence $\hat \phi_n\rightarrow \hat \phi$ in $C^2_{loc}(\mathbb R^2)$ with $\hat \phi$ satisfying
    \begin{equation*}
    \begin{split}
    &\Delta \hat \phi+e^{\hat u}(1-2e^{\hat u})\hat \phi=0,\quad\text{in}\quad \mathbb R^2,\\
    &|\hat \phi(0)|=\|\hat \phi\|_{L^\infty}=1.
    \end{split}
    \end{equation*}
    Here $\hat u$ is the unique topological solution of
    \begin{equation*}
    \begin{cases}
    \Delta \hat u+e^{\hat u}(1-e^{\hat u})=4\pi\sum_{i=1}^{l_{1,1}} \delta_{\hat p_i},\quad \text{in}\quad \mathbb R^2,\\
    \int_{\mathbb R^2}(1-e^{\hat u})^2<\infty,\quad \hat u<0.
    \end{cases}
    \end{equation*}
    \end{itemize}
    In order to obtain contradiction, we need to show $\hat \phi\in H^2(\mathbb R^2)$ both in Case (I) and (II). Let $\eta(x)\in C_c^\infty(\mathbb R^2)$, $\eta\equiv 1$ in $B_1(0)$, $\eta \equiv 0$ in $B_2^c(0)$, $\eta_R(x)=\eta(\frac xR)$. Taking $\eta_R^2\hat\phi$ as a test function, we get
    \begin{equation*}
    \int_{\mathbb R^2}\eta_R^2|\nabla \hat \phi|^2+e^{2\hat u}\eta_R^2\hat\phi^2\leq \frac 12 \int_{\mathbb R^2}\eta_R^2|\nabla \hat \phi|^2+\frac C{R^2}\int_{B_{2R}\backslash B_R}\hat \phi^2+\int_{\mathbb R^2} e^{\hat u}(1-e^{\hat u})\eta_R^2\hat\phi^2\leq C.
    \end{equation*}
    This implies $\hat \phi\in H^1(\mathbb R^2)$. Taking $\eta_R^2\Delta \hat \phi$ as a test function, one gets,
    \begin{equation*}
    \begin{split}
    \int_{\mathbb R^2}\eta_R^2|D_{ij}\hat \phi|^2&=\int_{\mathbb R^2}2\eta_RD_i\eta_RD_i\hat \phi D_{jj}\hat\phi-\int_{\mathbb R^2}2\eta_RD_j\eta_RD_i\hat \phi D_{ij}\hat\phi+\int_{\mathbb R^2}e^{\hat u}(1-2e^{\hat u})\eta_R^2\hat \phi\Delta\hat\phi\\
    &\leq \frac 12\int_{\mathbb R^2}\eta_R^2|D_{ij}\hat \phi|^2 +C\leq 2C.
    \end{split}
    \end{equation*}
    Now we have $\hat\phi\in H^2(\mathbb R^2)$,  by Theorem B, $\hat \phi\equiv 0$ which contradicts to $|\hat \phi(0)|=1$.
\section{Construction of topological solutions}
Suppose $\psi(x)$ is a topological solution of
\begin{equation}
\label{4.1}
\Delta \psi+e^{\psi}(1-e^{\psi})=4\pi\sum_{i=1}^l\alpha_i\delta_{p_i},\quad \text{in}\quad \mathbb R^2.
\end{equation}
The corresponding linearized operator $L$ is $\Delta +e^\psi(1-2e^{\psi})$ and we assume that
\begin{equation}
\|L h\|_{L^2(\mathbb R^2)}\geq C\|h\|_{H^2(\mathbb R^2)},\quad \forall h\in H^2(\mathbb R^2) \quad \text{for some constant }C>0.
\end{equation}
Also by Han \cite{Han2000}, for $R>\max |p_i|+1$, we have
\begin{equation}
\label{to1}
|\psi(x)|+|\nabla \psi(x)|\leq c_1e^{-c_2|x|}, \quad \forall |x|\geq R,\quad \text{for some constants } c_1,c_2>0.
\end{equation}
Consider a cut-off function $\eta(x)$ with $\eta(x)\equiv 1$ in $B_\delta(0)$, $\eta(x)\equiv 0$ in $B_{2\delta}^c(0)$ for some $\delta>0$ small and $F_\varepsilon(v)$ as follows
\begin{equation}
\label{4.3}
F_\varepsilon(v)=\Delta v+\frac 1{\varepsilon^5} e^{\eta\psi_\varepsilon+\varepsilon^3 v}(1-e^{\eta\psi_\varepsilon+\varepsilon^3 v})-\frac{\eta}{\varepsilon^5}e^{\psi_{\varepsilon}}(1-e^{\psi_\varepsilon})+\varepsilon^{-3}(2\nabla\eta\cdot\nabla\psi_\varepsilon+\psi_\varepsilon\Delta\eta),
\end{equation}
where $\psi_\varepsilon(x)=\psi(\frac x{\varepsilon})$. By a direct computation, it can be checked that $u_\varepsilon=\eta\psi_\varepsilon+\varepsilon^3 v_{\varepsilon}$ is a solution of
\begin{equation}
\label{4.4}
\Delta u_\varepsilon+\frac 1{\varepsilon^2}e^{u_{\varepsilon}}(1-e^{u_\varepsilon})=4\pi\sum_{i=1}^l \alpha_i\delta_{\varepsilon p_i},\quad \text{in}\quad \Omega
\end{equation}
provided $F_\varepsilon(v_\varepsilon)=0$ and $\varepsilon$ is small enough.
\begin{lemma}
\label{lem4.1}
There is a constant $\varepsilon_0>0$ such that if $0<\varepsilon<\varepsilon_0$, we have
\begin{itemize}
\item[(I)] $\|F_\varepsilon(0)\|_{L^2(\Omega)}\leq c_1e^{-c_2/\varepsilon}$ for some constants $c_1,c_2>0$.
\item[(II)] $DF_\varepsilon(0)$ is an isomorphism from $H^2(\Omega)$ onto $L^2(\Omega)$, moreover, we have
\begin{equation*}
\|DF_\varepsilon(0) h\|_{L^2(\Omega)}\geq C\|h\|_{H^2(\Omega)},\quad \forall h\in H^2(\Omega) \quad \text{for some constant }C>0.
\end{equation*}
\item[(III)] $\|DF_\varepsilon(v)h-DF_\varepsilon(0)h\|_{L^2(\Omega)}\leq C\varepsilon \|h\|_{H^2(\Omega)}$, for  $\|v\|_{H^2(\Omega)}\leq 1,$ $\forall h\in H^2(\Omega)$.
\end{itemize}
\end{lemma}
\begin{proof}
By definition of $F_\varepsilon(v)$, we have
\begin{equation*}
F_\varepsilon(0)=\frac 1{\varepsilon^5}(e^{\eta\psi_\varepsilon}(1-e^{\eta\psi_\varepsilon})-\eta e^{\psi_\varepsilon}(1-e^{\psi_\varepsilon}))+\frac 1{\varepsilon^3}(2\nabla\eta\cdot\nabla\psi_\varepsilon+\psi_\varepsilon\Delta\eta).
\end{equation*}
Since the support of $\nabla \eta$, $\Delta \eta$ is contained in $B_{2\delta}(0)\backslash B_\delta(0)$ and $\psi_\varepsilon(x),\nabla\psi_\varepsilon$ decay to $0$ exponentially fast as $e^{-c/\varepsilon}$ for $|x|\geq \delta$ by \eqref{to1}, we have
\begin{equation*}
\left|\frac 1{\varepsilon^3}(2\nabla\eta\cdot\nabla\psi_\varepsilon+\psi_\varepsilon\Delta\eta)\right|\leq c_1e^{-c_2/\varepsilon}.
\end{equation*}
Since $\eta\equiv 1$ in $B_\delta(0)$, the first term of $F_\varepsilon(0)$ vanishes in $B_\delta(0)$ and again by the exponential decay property of $\psi_\varepsilon$ in $|x|\geq \delta$, one get $\|F_\varepsilon(0)\|_{L^\infty(\Omega)}\leq c_1e^{-c_2/\varepsilon}$ which implies (I) is true.
\par We prove (II) by contradiction. Suppose there exists $h_n\in H^2(\Omega)$ such that
\begin{equation*}
\|h_n\|_{H^2(\Omega)}=1,\quad  \|DF_{\varepsilon_n}(0)h_n\|_{L^2(\Omega)}=o(1).
 \end{equation*}Set $\eta_1(x)=\eta(4x)$ and $\tilde h_n=(1-\eta_1)h_n$. Then $\tilde h_n$ solves
\begin{equation}
\label{4.5}
\Delta \tilde h_n-\frac 1{\varepsilon_n^2}e^{\eta\psi_\varepsilon}(2e^{\eta\psi_\varepsilon}-1)\tilde h_n=(1-\eta_1)DF_{\varepsilon_n}(0)h_n-2\nabla\eta_1\cdot\nabla h_n-\Delta \eta_1 h_n.
\end{equation}
Multiplying  \eqref{4.5} with  $\tilde h_n$ and integrating by parts, we get
\begin{equation}
\label{4.6}
\|\nabla \tilde h_n\|^2_{L^2(\Omega)}+\frac 1{\varepsilon_n^2}\|\tilde h_n\|^2_{L^2(\Omega)}\leq C\|\tilde h_n\|_{L^2(\Omega)}(\|h_n\|_{H^1(B^c_{\delta/4}(0))}+\|DF_{\varepsilon_n}(0)h_n\|_{L^2(B^c_{\delta/4}(0))})
\end{equation}
for some constant $C>0$. From  \eqref{4.6} and $\|h_n\|_{H^1(\Omega)}\leq 1$, one has $\|\tilde h_n\|_{L^2(\Omega)}\leq C\varepsilon_n^2$. Since $\eta_1\equiv 1$ in $B^c_{\delta/2}(0)$, \eqref{4.6} implies
\begin{equation}
\label{4.7}
\begin{split}
\|\nabla h_n\|_{L^2(B^c_{\delta/2}(0))}+\frac 1{\varepsilon_n}\|h_n\|_{L^2(B^c_{\delta/2}(0))}&\leq \|\nabla \tilde h_n\|_{L^2(\Omega)}+\frac 1{\varepsilon_n}\|\tilde h_n\|_{L^2(\Omega)}\leq C\varepsilon_n.
\end{split}
\end{equation}
Let $\eta_2(x)=\eta(2x)$ and $\bar h_n=(1-\eta_2)h_n$. Again, by the same arguments as $\tilde h_n$, we obtain
\begin{equation}
\label{4.8}
\|\nabla \bar h_n\|_{L^2(\Omega)}+\frac 1{\varepsilon_n}\|\bar h_n\|_{L^2(\Omega)}\leq C\varepsilon_n(\|h_n\|_{H^1(B^c_{\delta/2}(0))}+\|DF_{\varepsilon_n}(0)h_n\|_{L^2(B^c_{\delta/2}(0))})=o(\varepsilon_n).
\end{equation}
By \eqref{4.8}, it follows that $\|\bar h_n\|_{L^2(\Omega)}=o(\varepsilon_n^2)$ which implies $\|\Delta \bar h_n\|_{L^2(\Omega)}=o(1)$. Then by $W^{2,2}$ estimates and $1-\eta_2\equiv 1$ in $B^c_{\delta}(0)$, we have $\|h_n\|_{H^2(B^c_{\delta}(0))}=o(1)$ which implies $\|h_n\|_{H^2(B_\delta(0))}=1+o(1)$.
\par  Set $\hat h_n(x)=\eta(\varepsilon_n x) h_n(\varepsilon_n x)$ and $\eta_n(x)=\eta(\varepsilon_n x)$. By a direct computation, $\hat h_n$ satisfies
\begin{equation}
\label{4.9}
\begin{split}
\Delta \hat h_n+e^{\psi}(1-2e^{\psi})\hat h_n=&2\nabla\hat h_n\cdot \nabla \eta_n+\hat h_n\Delta \eta_n+\varepsilon_n^2\eta_n[DF_{\varepsilon_n}(0)h_n](\varepsilon_n x)\\
&+[e^{\eta_n\psi}(2e^{\eta_n\psi}-1)-e^{\psi}(2e^{\psi}-1)]\hat h_n.
\end{split}
\end{equation}
Set $\hat \Sigma_0=\{\frac \delta{\varepsilon_n}\leq |x|\leq \frac {2\delta}{\varepsilon_n}\}$, $\Sigma_0=\{\delta\leq |x|\leq 2\delta\}$. Since the last term in \eqref{4.9} vanishes in $\hat\Sigma_0^c$, then by \eqref{4.8} and our assumption,
\begin{equation*}
\begin{split}
\|\hat h_n\|_{H^2(\mathbb R^2)}&\leq C(\varepsilon_n\|\nabla \hat h_n\|_{L^2(\hat \Sigma_0)}+\|\hat h_n\|_{L^2(\hat \Sigma_0)}+\varepsilon_n\|DF_{\varepsilon_n}(0)h_n\|_{L^2(B_{2\delta}(0))})\\
&\leq C(\varepsilon_n\|\nabla h_n\|_{L^2(\Sigma_0)}+\varepsilon_n^{-1}\|h_n\|_{L^2(\Sigma_0)}+\varepsilon_n\|DF_{\varepsilon_n}(0)h_n\|_{L^2(B_{2\delta}(0))})=o(\varepsilon_n).
\end{split}
\end{equation*}
This contradicts to $\|h_n\|_{H^2(B_\delta(0))}=1+o(1)$. The second part of (II) follows immediately. Since $DF_{\varepsilon}(0)$ is a self-adjoint operator from $H^2(\Omega)\rightarrow L^2(\Omega)$, we obtain that $DF_\varepsilon$ is an isomorphism from $H^2(\Omega)$ to $L^2(\Omega)$.

\par The estimate for $\|DF_\varepsilon(v)h-DF_\varepsilon(0)h\|_{L^2(\Omega)}$ follows from
\begin{equation*}
DF_\varepsilon(v)h-DF_\varepsilon(0)h=\frac 1{\varepsilon^2}e^{\eta\psi_\varepsilon}(e^{\varepsilon^3 v}-1)h-\frac 2{\varepsilon^2}e^{2\eta\psi_\varepsilon}(e^{2\varepsilon^3 v}-1)h
\end{equation*}
and the embedding of $H^2(\Omega)\hookrightarrow L^\infty(\Omega)$.
\end{proof}
\textbf{Proof of Theorem \ref{thm1.3}. }
We define a functional $G_\varepsilon: H^2(\Omega)\rightarrow L^2(\Omega)$ by
\begin{equation}\label{4.10}
G_\varepsilon(v)=v-[DF_\varepsilon(0)]^{-1}F_\varepsilon(v).
\end{equation}
For any fixed point $v_\varepsilon$ of $G_\varepsilon$ in $H^2(\Omega)$, $u_\varepsilon=\eta\psi_\varepsilon+\varepsilon^3 v_\varepsilon$ is a solution of \eqref{1.13}. It suffices to prove $G_\varepsilon$ admits a fixed point in $H^2(\Omega)$ for $\varepsilon>0$ small. Set $\mathcal B=\{v\in H^2(\Omega)||v|_{H^2(\Omega)}\leq 1\}$.
\par Claim: $G_\varepsilon:\mathcal B\rightarrow \mathcal B$ is a well-defined contraction mapping provided $\varepsilon>0$ is small enough.
\par It follows from Lemma \ref{lem4.1} that, for $\varepsilon>0$ small enough,
\begin{equation*}
\begin{split}
\|DG_\varepsilon(v)h\|_{L^2(\Omega)}\leq \|[DF_\varepsilon(0)]^{-1}\|\|(DF_\varepsilon(v)-DF_\varepsilon(0))h\|_{L^2(\Omega)}\leq C\varepsilon|h|_{H^2(\Omega)}
\end{split}
\end{equation*}
for $v\in\mathcal B$. Then by Lemma \ref{lem4.1}, we have
\begin{equation*}
\|G_\varepsilon(0)\|_{H^2(\Omega)}\leq C\|F_\varepsilon(0)\|_{L^2(\Omega)}\leq Ce^{-c/\varepsilon}.
\end{equation*}
For any $v_1,v_2\in\mathcal B$, for $\varepsilon>0$ small enough, we have
\begin{equation*}
\begin{split}
\|G_\varepsilon(v_1)\|_{H^2(\Omega)}&\leq \|G_\varepsilon(0)\|_{H^2(\Omega)}+\|G_\varepsilon(v_1)-G_\varepsilon(0)\|_{H^2(\Omega)}\\
&\leq \|G_\varepsilon(0)\|_{H^2(\Omega)}+(\sup_{v\in \mathcal B}\|DG_\varepsilon(v)\|)\|v_1\|_{H^2(\Omega)}\leq C(e^{-c/\varepsilon}+\varepsilon),
\end{split}
\end{equation*}
and
\begin{equation*}
\|G_\varepsilon(v_1)-G_\varepsilon(v_2)\|_{H^2(\Omega)}\leq (\sup_{v\in \mathcal B}\|DG_\varepsilon(v)\|)\|v_1-v_2\|_{H^2(\Omega)}\leq C\varepsilon\|v_1-v_2\|_{H^2(\Omega)}.
\end{equation*}
Hence, if $\varepsilon>0$ is small enough, $G_\varepsilon:\mathcal B\rightarrow \mathcal B$ is a well-defined contraction mapping, which implies $G_\varepsilon$ has a unique fixed point $v_\varepsilon\in \mathcal B$. By a direct computation, one can see that $u_\varepsilon=\eta\psi_\varepsilon+\varepsilon^3 v_\varepsilon$ satisfies \eqref{1.13}. Also from Sobolev embedding $H^2(\Omega)\hookrightarrow L^\infty(\Omega)$ and  $\psi_\varepsilon$ decaying exponentially to $0$, $\forall |x|\geq r$, $\forall r>0$, we see $u_\varepsilon$ is a topological solution.
\section*{Acknowledgement}
The first author would like to thank  Center for Advanced Study in Theoretical Sciences, National Taiwan University for the warm hospitality when this work was done. The work of the first author is partially supported by NSFC-11401376 and China Postdoctoral Science Foundation 2014M551391.

\providecommand{\bysame}{\leavevmode\hbox to3em{\hrulefill}\thinspace}
\providecommand{\MR}{\relax\ifhmode\unskip\space\fi MR }
\providecommand{\MRhref}[2]{%
  \href{http://www.ams.org/mathscinet-getitem?mr=#1}{#2}
}
\providecommand{\href}[2]{#2}

\end{document}